\documentclass{article}
\usepackage{natbib}
\usepackage{amsmath,amsfonts,amsthm}
\usepackage{graphicx}
\usepackage{dsfont} % indicator function
\usepackage{enumerate}
\usepackage[titletoc,title]{appendix}
\newtheorem{proposition}{Proposition}

\newtheorem{corollary}[proposition]{Corollary}

\newtheorem*{remark*}{Remark}
\newtheorem*{proposition*}{Proposition}
\newtheorem*{acknowledgement*}{Acknowledgment}

\title{Two identities involving counts of binary matrices}
\author{Hannes Leeb\\
(Department of Statistics, University of Vienna)}
\date{\today}

\begin{document}
\sloppy
\maketitle

\begin{abstract}
In the context of generating uniform random contingency tables
with pre-specified marginals, the number of (binary) matrices with
given row- and column-sums is a well-studied object in the literature.
We will denote this number by $N(p,q)$, where $p$ and $q$ are
the vectors of row- and column-sums.
The existing literature is mainly focused on computing or approximating
$N(p,q)$.
In this paper, we present two identities for polynomials whose 
coefficients depend on the $N(p,q)$ and explore some
consequences.
\end{abstract}

There is a considerable body of literature regarding
the number of matrices with pre-specified marginals,
ranging from early works like~\cite{Goo77a} to
more recent ones such as~\cite{Arm24a} (see also the references
cited in these papers). Most of these works focus on counting the
number of matrices with given marginals, because these counts are
instrumental for generating uniform random contingency tables.
In this paper, we study a different aspect of these numbers. In particular,
we provide two identities for polynomials whose coefficients
depend on $N(p,q)$ and explore some of their consequences; 
cf. Proposition~\ref{p1} and Corollaries~\ref{c1} to~\ref{c3}. 
While the first two corollaries appear to be new, their practical
relevance is unclear. In the last corollary, Proposition~\ref{p1} is
used to obtain an efficient bound for a particular symmetric polynomial
that occurs in a problem in statistics. Finally, we point out that
Proposition~\ref{p1} is also related to Stirling numbers of the second kind.

Before we can state our result, some notation is required:
For a sequence $p=(p_1,p_2,\dots)$ of non-negative integers,
set $|p| = \sum_{i\geq 1} p_i$ and let 
$[p] = (\#\{j: p_j=i\})_{i\geq 1}$ count the number of occurrences
of $i\geq 1$ in $p$, so that $|[p]|$ is the number of non-zero
elements in, or the `length' of,  $p$ .
For $p\in \mathbb N_0^\infty$ and $x\in \mathbb R$, we write 
$p!$ and $x^{\underline{p}}$
for the product of factorials $\prod_{i\geq 1}p_i!$
and of falling factorials $\prod_{i\geq 1} x^{\underline{p_i}}$,
respectively.
If $p$ is an integer partition of $m$, i.e., if $|p|=m$ and
$p_1\geq p_2 \geq \dots$, we write $p \vdash m$.
Given two integer partitions $p$ and $q$ of $m$, we denote
the number of $|[p]|\times |[q]|$ binary matrices
whose row- and column-sums equal $p$ and $q$, respectively,
by $N(p,q)$ (for $m=0$, this number is to be interpreted as one).

\begin{proposition}\label{p1}
For $q \vdash m$ and $x\in \mathbb R$, we have
$$
	\sum_{p \vdash m}
	 \frac{x^{\underline{|[p]|}}}{[p]!}
	N(p,q) \quad=\quad
	\frac{x^{\underline{q}}}{q!}.
$$
Moreover, for $x, y\in\mathbb R$, we have
$$
\sum_{p,q \vdash m} 
\frac{x^{\underline{|[p]|}}}{[p]!}
\frac{y^{\underline{|[q]|}}}{[q]!}
N(p,q)
\quad=\quad \frac{(x y)^{\underline{m}}}{m!}.
$$
\end{proposition}

\begin{proof}
For the first equality, 
write $P(x)$ and $Q(x)$ for the 
left-hand side and the right-hand side, respectively.
It is easy to see that $P(k)=Q(k)$ for
$k=1,\dots,m+1$, say, because both count the number of binary
matrices with $k$ rows and with column-sums given by $q$.
Since $P(x)$ and $Q(x)$ both are polynomials in $x$ for degree
$m$, it follows that $P(x)=Q(x)$ for each $x\in \mathbb R$.

For the second equality, it suffices to show that
$$
\sum_{q \vdash m} 
\frac{x^{\underline{q}}}{q!}
\frac{y^{\underline{|[q]|}}}{[q]!}
\quad=\quad \frac{(x y)^{\underline{m}}}{m!}.
$$
Write $P(x,y)$ and $Q(x,y)$ for the expressions on the left and on the
right, respectively, in the preceding display. For integers $k$ and $l$,
both $P(k,l)$ and $Q(k,l)$ count the number of $k\times l$ binary matrices
with $m$ $1$'s. For fixed $k$, we thus have
$P(k,l) = Q(k,l)$ for each integer $l$. Because $P(k,y)$ and $Q(k,y)$
are polynomials in $y$ of degree $m$, it follows that $P(k,y) = Q(k,y)$
for each integer $k$ and each $y\in \mathbb R$. For fixed $y$, both
$P(x,y)$ and $Q(x,y)$ are polynomials in $x$ of degree $m$, so that
$P(x,y) = Q(x,y)$, as claimed.
\end{proof}

\begin{remark*} \normalfont
(i)
Using the algebraic definition of the multinomial coefficient, i.e.,
using falling factorials, the result can also be written as
\begin{align*}
\sum_{p\vdash m} {x \choose [p]} N(p,q)
&\quad=\quad 
\prod_{i\geq 1}{x \choose q_i}\quad\text{and}\\
\sum_{p,q\vdash m}{x \choose [p]}{y\choose [q]} N(p,q)
&\quad=\quad
{x y \choose m}.
\end{align*}
(ii)
The combinatorial proof given above is straight-forward once the desired
statements are written as in the proposition. 
The first statement was originally derived using a 
a less elegant algebraic proof that relies on 
`conditioning on', or fixing, the first column of a
matrix counted in $N(p,q)$ (a technique also used by~\cite{Mil13a}). 
That proof crucially relies on the fact that
$x^{\underline{l+k}}/x^{\underline{l}} = (x-l)^{\underline{k}}$,
suggesting that it may be difficult to derive identities similar to
those in Proposition~\ref{p1} with the falling factorial 
$x^{\underline{l}}$ replaced by another function of $x$ and $l$.
\end{remark*}

\begin{corollary}\label{c1}
For each $q \vdash m$, we have
\begin{align*}
\sum_{p\vdash m} (-1)^{|[p]|} {|[p]| \choose [p]} N(p,q)
&\quad=\quad (-1)^m\quad\text{and}
\\
\sum_{p,q\vdash m} (-1)^{|[p]|+|[q]|} 
{|[p]| \choose [p]}
{|[q]| \choose [q]}
N(p,q)
&\quad=\quad 
\frac{1^{\underline{m}}}{m!}.
\end{align*}
Note that the right-hand side of the second equality equals one
if $m=1$ and zero if $m>1$.
\end{corollary}
\begin{proof}
Use Proposition~\ref{p1} with $x,y=-1$ and simplify.
\end{proof}

The following result exploits the fact that, in the second identity in
Proposition~\ref{p1}, the expression on the left is a polynomial in $x$ and
$y$ while the expression on the right is a polynomial in $x y$.
\begin{corollary}\label{c2}
For $1\leq i \leq m$ and $1\leq j \leq m$, we have
$$
\sum_{k=i}^m\sum_{l=j}^m
s(k,i) s(l,j) \sum_{\stackrel{p,q \vdash m}{|[p]|=k, |[q]|=l}}
\frac{N(p,q)}{[p]![q]!}
\quad=\quad
\left\{
\begin{tabular}{c l}
$0$ &\text{ if $i \neq j$ and}\\
$\frac{s(m,i)}{m!}$ & \text{ if $i=j$,}
\end{tabular}
\right.
$$
where $s(\cdot,\cdot)$ denotes the (signed) Stirling number of the first kind.
\end{corollary}
\begin{proof}
It is elementary to verify that the triple-sum is the
coefficient of $x^i y^j$ on the left-hand side of the second identity
in Proposition~\ref{p1}, while the right-hand side of that identity
is just $\sum_{i=1}^m \frac{s(m,i)}{m!} (xy)^i$.
\end{proof}

Our next corollary is concerned with a particular symmetric
polynomial $R_m$ that occurs in a problem in statistics; cf.~\cite{Lee25a}.
In this problem, (upper and lower) bounds for the arguments
of the polynomial are available and
a bound on $R_m$ is desired.
Because $R_m$ is an alternating sum with many terms and large coefficients, 
the triangle inequality gives poor bounds.
Our results allow us to re-write $R_m$ in a form that is much easier to bound.
For properties of symmetric functions, 
in particular the elementary symmetric functions $e_p(\cdot)$
and of the monomial symmetric functions $m_q(\cdot)$ that we use in
the following, we refer to~\cite{Sta99a}.

\begin{corollary}\label{c3}
Define
$$
R_m\quad=\quad
\sum_{p \vdash m}
\frac{(2 |[p]|-1)!!}{(-2)^{|[p]|}}
\frac{1}{[p]!} e_p(\mu_1,\dots, \mu_k),
$$
where $e_p(\cdot)$ denotes the elementary symmetric function corresponding
to $p$. Then
$$
R_m\quad=\quad  \frac{1}{(-2)^{m} m!}
\mathbb E\left[ \left(
Z_1^2 \mu_1 + \dots + Z_k^2 \mu_k
\right)^m\right],
$$
where the $Z_1,\dots, Z_k$ are independent and identically distributed
standard Gaussian random variables.
\end{corollary}
\begin{proof}
Abbreviate $e_p(\mu_1,\dots, \mu_k)$ and
$m_q(\mu_1,\dots, \mu_k)$ by $e_p$ and $m_q$, respectively.
For $p\vdash m$, $e_p$ can 
be expressed in terms of the $m_q$'s as
$$
e_p\quad=\quad \sum_{q \vdash m} N(p,q) m_q.
$$
Plugging this into the formula for $R_m$, we see that
$R_m$ is given by
\begin{align*}
& \sum_{p\vdash m} 
\frac{(2 |[p]|-1)!!}{(-2)^{|[p]|}}
\frac{1}{[p]!} \sum_{q\vdash m} N(p,q) m_q
\quad=\quad
 \sum_{q\vdash m} \sum_{p\vdash m}
\frac{(2 |[p]|-1)!!}{(-2)^{|[p]|}}
\frac{1}{[p]!} \sum_{q\vdash m} N(p,q) m_q 
\\
&=\quad 
 \sum_{q\vdash m} (-2)^{-m} \frac{(2 q-1)!!}{q!} m_q
\quad=\quad\frac{1}{(-2)^{m} m!}
\sum_{q \vdash m} {m \choose q} (2 q - 1)!! m_q,
\end{align*}
where the second equality follows from the first relation in 
Proposition~\ref{p1}
with $x=-1/2$.
Now $\mathbb E[(Z_1^2 \mu1 + \dots +Z_k^2 \mu_k)^m]$ equals
\begin{align*}
&\mathbb E\left[
\sum_{q\vdash m} {m \choose q} m_q(Z_1^2\mu_1,\dots, Z_k^2 \mu_k) \right]
\quad=\quad
\sum_{q\vdash m} {m \choose q} 
\mathbb E[m_q(Z_1^2\mu_1,\dots, Z_k^2 \mu_k)]\\
&=\quad \sum_{q\vdash m} {m \choose q} (2 q-1)!! m_q(\mu_1,\dots, \mu_k),
\end{align*}
where the last equality follows from the definition of $m_q$ and
the fact that, for any sequence $(\alpha_1,\dots, \alpha_k,0,\dots)$
with $[\alpha]=[q]$,  we have
$\mathbb E[ (Z_1^2 \mu_1)^{\alpha_1} \cdot \dots \cdot (Z_k^2 \mu_k)^{\alpha_k}]
= (2 q-1)!! \mu_1^{\alpha_1} \cdot\dots\cdot \mu_k^{\alpha_k}$
because the $Z_i$'s are i.i.d. with $\mathbb E[Z_1^{2 q_i}] = (2 q_i-1)!!$.
\end{proof}

Lastly, we point out a connection between our result and the
Stirling numbers of the second kind.
These numbers are defined through
the relation $\sum_{l=1}^m x^{\underline l} S(m,l) = x^m$.
Interestingly, this equality can also be derived from the first 
relation in Proposition~\ref{p1}:
Consider $q=(1,1,\dots,1,0,0,\dots) \vdash m$. Using the proposition
with this particular $q$ gives
\begin{align*}
\sum_{l\geq 1}
x^{\underline{l}}
\sum_{\stackrel{p\vdash m}{|[p]|=l}}
\frac{1}{[p]!} N(p,q)
\quad =\quad x^m.
\end{align*}
A little reflection shows, for $|[p]|=l$, that
$\frac{1}{[p]!} N(p,q)$ is just the number
of distinct partitions of an $m$-set into $l$  sets of size
$p_1,p_2,\dots, p_l$, respectively, so that the inner sum in the preceding
display is $S(m,l)$.

\begin{acknowledgement*}\normalfont
The code provided by \cite{Mil13a} for computing $N(p,q)$ has
been very helpful in preliminary investigations.
Also, helpful comments by Christian Krattenthaler on an earlier
version of the manuscript are greatly appreciated.
\end{acknowledgement*}

\bibliographystyle{apalike}
\bibliography{lit}

\end{document}